\def\filedate{1 September 2007}

\def\filestatus{4th draft}
\documentclass{lms}

\usepackage{amsfonts}
\usepackage{amssymb,amsmath}
\usepackage[mathscr]{eucal}
\usepackage{latexsym}

\newcommand{\Sym}{\mathrm{Sym}}\newcommand{\Alt}{\mathrm{Alt}}

 \providecommand{\ker}{\mathrm{Ker\ }}

\providecommand{\det}{\mathrm{det\ }} 
\providecommand{\dim}{\mathrm{dim\ }} 
\providecommand{\lcm}{\mathrm{lcm}}
\providecommand{\hcf}{\mathrm{hcf}}

\newtheorem{theorem}{Theorem}
\newtheorem{lemma}[theorem]{Lemma}
\newtheorem{proposition}[theorem]{Proposition}
\newnumbered{definition}[theorem]{Definition}
\renewcommand{\theenumi}{(\roman{enumi})}

\DeclareMathOperator{\sgn}{sgn}

\simpleequations

\title[On the distribution of conjugacy classes]{On the distribution of conjugacy classes between the
cosets of a finite group in a cyclic extension}

\author{John R. Britnell}\author{John R. Britnell \and Mark Wildon}

\classno{20D60, 20E45 (primary)}

\extraline{\filestatus, \filedate.}

\begin{document}
\maketitle

\begin{abstract}
Let $G$ be a finite group and~$H$ a normal subgroup
such that~$G/H$ is cyclic. Given a conjugacy class~$g^G$ 
of~$G$ we define its centralizing subgroup
to be~$HC_G(g)$. Let~$K$ be such
that $H\le K\le G$. We show that the $G$-conjugacy classes 
contained in $K$ whose centralizing subgroup is~$K$, are
equally distributed between the cosets of~$H$ in~$K$.
The proof of this result is entirely elementary.
As an application we find expressions for the number of conjugacy 
classes of~$K$ under its own action, in terms of quantities 
relating only to the action of~$G$. 
\end{abstract}

\section{Introduction}
Let $G$ be a finite group and $H$ a normal subgroup
such that the quotient~$G/ H$ is cyclic. In this paper
we establish a quite general result (Theorem \ref{MT}) about
the distribution of the conjugacy classes of $G$ between
the cosets of $H$. A key idea in this work is that of the
centralizing subgroup of a conjugacy class; the centralizing
subgroup of $g^G$ is defined
to be the smallest subgroup of~$G$ containing both~$H$ and the 
centralizer~$C_G(g)$.
This subgroup 
determines how the class splits when the
conjugacy action is restricted to subgroups of $G$ containing~$H$.
We demonstrate that 
the centralizing subgroup is fundamental
to an understanding of the distribution of the conjugacy classes of $G$.


Theorem~\ref{MT}
states that the conjugacy classes with a particular centralizing subgroup~$K$ are
equally distributed amongst the cosets of $H$ in $K$; the proof occupies the greater part
of the paper. In Section \ref{application} we present an interesting application of Theorem~\ref{MT}: 
enumerating the conjugacy classes of a subgroup $K$ in the range $H\le K\le G$ in terms of the numbers of conjugacy classes of $G$ contained in various subgroups.

The reader may recognize that our main result has a character-theoretic flavour.
Indeed it seems likely that Theorem~\ref{MT} can
be proved by character theory: specifically by means of a combination of Clifford Theory
and Brauer's Permutation Lemma (for an account of these subjects see \cite{Isaacs}). 
So far as the authors have investigated, it appears unlikely
that such an approach will lead to
a shorter proof than the elementary one given here.

The proof of Theorem \ref{MT} relies on a preliminary result (Lemma \ref{PL}), which
states that the number of conjugacy classes in a generating coset of $G/H$ is equal to the number of
conjugacy classes of $G$ contained in $H$ which do not split when the action by conjugacy is restricted to $H$.
Although the proof of this result is straightforward, we are not aware of any previous appearance of the
fact in the literature, at least in this general form. The special case where $|G/H|=2$ is of course
well known and often cited, partly because of its usefulness in deriving the character table of $\Alt(n)$
from that of $\Sym(n)$; see~\cite{JL} for example. 

In \cite{BW1} the authors present a result which relies upon a special case of Lemma~\ref{PL}, together 
with Hall's Marriage Theorem. 
It is shown, in the case where $|G/H|$ is prime, that the set of 
conjugacy classes whose centralizing group is $G$ can be partitioned in such a way that each part 
contains one class from
each coset, and any two classes in the same part contain elements which commute with one another.

The importance of Lemma~\ref{PL} in the present paper is that it can be used to derive 
a set of linear equations which relate the numbers of conjugacy classes in
different cosets of $H$ which have a given centralizing subgroup. This allows us to reduce the problem to one of linear
algebra: namely, finding the dimension of one of the eigenspaces of a certain matrix. A further reduction of the
problem by means of a tensor factorization allows us to focus on the case where $G/H$ is a cyclic $p$-group; in
this form, the problem turns out to be readily soluble.

From this brief description of the proof, it will be clear to the reader that the proof is
to be presented backwards. Rather than building up to the main theorem, we shall proceed by reducing it
by stages to a simpler problem. Our justification for this \emph{modus operandi}, if one is needed,
is that it seems the most---perhaps the only---coherent way to present the argument.

We have not attempted to deal with cases where $G/H$ is non-cyclic. That the quotient should be abelian is
necessary (and sufficient) for each conjugacy class to lie wholly within a single coset. The case of a
non-cyclic, abelian quotient seems problematic however; for example, 
if~$G$ is nilpotent of class~$2$, and~$H$ is the centre of~$G$, then
the elements of $G$ whose centralizing subgroup is~$G$ are precisely the elements of~$H$.\footnote{The
authors would like to thank Peter~Neumann for this observation.} This precludes 
the possibility of
a result directly analogous to Theorem~\ref{MT}; however our methods
do seem to provide some information about the general case, and the
problem is surely worthy of further study. 

Throughout this paper we adopt the convention that a summation sign indicates a sum over a single variable,
which is in every case the variable denoted by the first letter appearing in the conditions below the sign.

\section{Statement of the main theorem and the principal lemma}
Throughout this paper, we shall assume that $G$ is a finite group, and that $H$ is a normal subgroup of
$G$ such that $G/H$ is cyclic.

\begin{definition*}
For an element $g\in G$, we define the \emph{centralizing subgroup} $\Delta_g$ of~$g$
with respect to $H$, to be $HC_G(g)$. For each conjugacy class $X$ of $G$,
we define the centralizing subgroup $\Delta_X$ to be $\Delta_g$ for an element $g\in X$.
\end{definition*}

\noindent In the present section we consider centralizing subgroups only with respect to $H$, and shall not
always mention $H$ explicitly. Later in the paper, however, we shall have occasion to refer to centralizing
subgroups with respect to other subgroups of $G$.

We can now state our main theorem.

\begin{theorem}\label{MT}
Let $G$ be a finite group, and let $H$ be a normal subgroup of $G$ such that~$G/H$ is cyclic. Let~$K$ be such
that $H\le K\le G$. Then the $G$-conjugacy classes contained in $K$ whose centralizing subgroup is $K$, are
equally distributed between the cosets of $H$ in $K$.
\end{theorem}

Before introducing the lemma which will be the principal tool in the proof of Theorem \ref{MT}, it
is convenient to make the following definition.
\begin{definition*}
We define an \emph{integral $(G,H)$-class} to be a conjugacy
class of $G$ whose centralizing subgroup with respect to $H$ is
$G$.
\end{definition*}
\noindent Equivalently, an integral $(G,H)$-class is one
which does not split when the conjugacy action is restricted to $H$.

\begin{lemma}\label{PL}
Suppose that the coset $Hx$ is a generator of the quotient group~$G/H$.
Then the number of conjugacy classes contained in $Hx$ is equal to the
number of integral~$(G,H)$-classes contained in $H$.
\end{lemma}
\begin{proof} If $h\in H$ and $\Delta_h=G$, then $C_G(h)$ meets every coset of $H$ in $G$,
and in particular it meets $Hx$. Now the number of integral $(G,H)$-classes in $H$ is
\begin{eqnarray*}
\sum_{\stackrel{\scriptstyle{h\in H}}{\scriptstyle{\Delta_h=G}}}\frac{1}{|h^G|}
& = & \frac{1}{|G|}\sum_{\stackrel{\scriptstyle{h\in H}}{\scriptstyle{\Delta_h=G}}}
|C_G(h)| \\
& = & \frac{1}{|H|}\sum_{h\in H}
|C_{Hx}(h)| \\
& = & \frac{1}{|H|}\sum_{g \in Hx}|C_H(g)| \\
& = & \frac{1}{|G|}\sum_{g\in Hx}|C_G(g)| \\
& = & \sum_{g\in Hx}\frac{1}{|g^G|},
\end{eqnarray*}
which is the number of conjugacy classes in $Hx$. This establishes the lemma. \end{proof}

We now want to widen our focus in two respects: by including conjugacy classes whose
centralizing group is a proper subgroup of $G$, and by considering all of the cosets of
$H$ in $G$. The following lemma lays the foundations.

\vbox{
\begin{lemma}\label{TL}
Suppose that $Hx$ and $Hy$ have equal order in the quotient
group $G/H$. Then there exists a permutation $\sigma$ of the elements
of $G$ such that
\begin{enumerate}
\item every subgroup of $G$ is $\sigma$-invariant,
\item $g_1 \sigma\, g_2 \sigma = g_2 \sigma\, g_1 \sigma$ if and only if $g_1g_2=g_2g_1$,
\item $\sigma$ permutes the conjugacy classes of $G$,
\item $(Hx)\sigma=Hy$.
\end{enumerate}
\end{lemma}
\begin{proof} Since $Hx$ and $Hy$ have equal order in $G/H$, there exists
an integer $a$, coprime with~$|G/H|$, such that $(Hx)^a\subseteq Hy$.
Now we may suppose that $a$ is also coprime with $|H|$, and hence that
$a$ is invertible modulo $|G|$. Therefore the map $g\mapsto g^a$
is a permutation of $G$, and clearly has the properties claimed in the lemma. \end{proof}}

The number of conjugacy classes with a particular centralizing subgroup is equal in
cosets of equal order in~$G/H$, since a conjugacy class in~$Hx$ is mapped by~$\sigma$
to a conjugacy class in~$Hy$, while the centralizing subgroup is invariant.
It follows that we lose nothing by selecting a representative coset of each
order in $G/H$. The following definition takes advantage of this fact.

\vbox{
\begin{definition*}Let $n = |G / H |$.
\begin{enumerate}
\item For a divisor $d$ of $n$, we define $K_d$
to be the unique subgroup of $G$ which contains $H$ as a subgroup
of index $d$.
\item For a divisor $d$ of $n$, we define
$\Gamma_d$ to be a representative coset of order $d$ in $G/H$. (So
$K_d$ is the subgroup generated by $\Gamma_d$.)
\item If $c|n$ and $d|c$, then we define
$N_d^c$ to be the number of conjugacy classes in~$\Gamma_d$ whose
centralizing subgroup is~$K_c$.
\end{enumerate}
\end{definition*}}

In terms of this new notation, we may restate Theorem \ref{MT} simply as follows.
\begin{theorem*}
Suppose that $c|n$ and $d|c$. Then $N_d^c=N_1^c$.
\end{theorem*}

\section{Reduction to linear algebra}

By multiple applications of Lemma \ref{PL} we derive a set of linear equations in the quantities~$N_d^c$,
allowing us to reduce the proof of Theorem \ref{MT} to a problem of linear algebra. To obtain the linear
equations, we need for each particular choice of~$(i,j)$, to express the following numbers in terms of the
quantities $N_d^c$:
\begin{enumerate}
\item the number $L_i^j$ of integral $(K_j, K_i)$-classes in $K_i$,
\item the number $R_i^j$ of integral $(K_j, K_i)$-classes in a generating coset of $K_j/K_i$.
\end{enumerate}
Lemma \ref{PL} tells us that these numbers are equal.

We first handle the quantity $L_i^j$. It is straightforward to identify the $G$-conjugacy classes of~$K_i$ whose
centralizing subgroups with respect to~$K_i$, contain $K_j$. But it is necessary to allow for the fact that
these classes may split when the conjugacy action is restricted to $K_j$. In fact, a $G$-conjugacy class in
$K_i$ with centralizing subgroup $K_c$ splits into $n/\lcm(j,c)$ classes under the action of~$K_j$. If $d$ is a
divisor of $i$, then $K_i$ contains $\phi(d)$ cosets of $H$ whose order in $G/H$ is $d$, where $\phi$ is Euler's
totient function; one of these is our representative coset $\Gamma_d$. Suppose that an element of $\Gamma_d$ has
centralizing subgroup $K_c$ with respect to $H$. Then its centralizing subgroup with respect to $K_i$ is
$K_{\lcm(i,c)}$. So the quantity $L_i^j$ is given by the formula
\begin{equation}\label{L}
L_i^j=\sum_{d|i}\phi(d)\sum_{\stackrel{\stackrel{\scriptstyle{c|n}}{\scriptstyle{d|c}}}{\scriptstyle{j|\lcm(i,c)}}}
\frac{n}{\lcm(j,c)}N_d^c.
\end{equation}

To find $R_i^j$ in terms of the quantities $N_d^c$, we need to describe a generating coset
of~$K_j/K_i$
as a union of cosets of~$H$. This requires us to look at the
arithmetic of $i$ and $j$ a little more closely.
\begin{lemma}\label{GCL}
Suppose that $i|j$, and let $i=uv$, where $v$ is the largest divisor
of $i$ coprime with~$j/i$. Suppose that $C_i$ is the cyclic subgroup of order $i$ inside a cyclic
group $C_j$ of order $j$. Then
\begin{enumerate}
\item any element of a generating coset of $C_j/C_i$ has order divisible by $j/v$,
\item if $d|v$, then the number of elements of order $jd/v$ in a generating coset is $u\phi(d)$.
\end{enumerate}
\end{lemma}

\begin{proof}
An element of order $k$ in $C_j$ is contained in a generating coset of $C_j/C_i$ if and only if $\lcm(i,k)=j$.
Suppose that this is the case. If $p$ is a prime divisor of $j$ which does not divide~$v$, then~$p^a$, the
highest power of~$p$ dividing~$j$, is strictly greater than the highest power of $p$ dividing $i$, and so $p^a$
must divide~$k$. Since no prime divisor of $j/v$ can divide $v$, it follows that~$j/v$ divides~$k$.

The number of elements of~$C_j$ with order~$k$ is~$\phi(k)$,
and the number of generating cosets of~$C_i$ in $C_j$ is $\phi(j/i)$, and so the number of
elements of order~$k$ in each such coset is $\phi(k)/\phi(j/i)$. Suppose that $k=dj/v$ where~$d$ is a divisor of $v$; then since $v$ is coprime with $j/v$ we see that
\[ \phi(k)/\phi(j/i)  = \phi(d)\phi(j/v)/\phi(j/i). \]
Now the prime divisors of $j/v$ are precisely the prime divisors of $j/i$.
Since $\phi(p^r)/\phi(p^s)=p^{r-s}$ for any prime $p$ and positive
integers $r \ge s$, it follows that $\phi(j/v)/\phi(j/i)=i/v=u$, which gives the required result.
\end{proof}

Now arguing similarly to the calculation of $L_i^j$ above, we find that
\[
R_i^j=\sum_{d|v}u\phi(d)\sum_{\stackrel{\stackrel{\scriptstyle{c|n}}{\scriptstyle{jd/v|c}}}{\scriptstyle{j|\lcm(i,c)}}}
\frac{n}{\lcm(j,c)}N_{jd/v}^c,
\]
where $u$ and $v$ are as in Lemma \ref{GCL}.
The condition that $j | \lcm(i,c)$ allows us to 
replace the condition that $jd/v | c$ with the simpler
condition that $d | c$. Hence
\begin{equation}\label{R}
R_i^j=\sum_{d|v}u\phi(d)\sum_{\stackrel{\stackrel{\scriptstyle{c|n}}{\scriptstyle{d|c}}}{\scriptstyle{j|\lcm(i,c)}}}
\frac{n}{\lcm(j,c)}N_{jd/v}^c,
\end{equation}
where $v$ is the greatest divisor of $i$ coprime with $j/i$, and $u=i/v$.

Lemma \ref{PL} gives us the following linear equation.
\begin{equation}\label{Omega}
\Omega_i^j:  \  L_i^j=R_i^j.
\end{equation}
Notice that when $j=i$, we have $v=i$ and $u=1$, and it is not hard to see that the
equation~$\Omega_i^j$ becomes trivial since the two sides are identical. If $j\neq i$ then~$\Omega_i^j$ is
non-trivial. 

We next make an observation which relates the set of equations $\{\Omega_i^j\}$ to the statement of
Theorem \ref{MT}.
\begin{proposition}\label{e1}
The linear equation $\Omega_i^j$ is satisfied if
$N_d^c=N_1^c$ for all $c|n$ and $d|c$.
\end{proposition}
\begin{proof*} The equation becomes
\[
\sum_{d|i}\phi(d)\sum_{\stackrel{\stackrel{\scriptstyle{c|n}}{\scriptstyle{d|c}}}{\scriptstyle{j|\lcm(i,c)}}}
\frac{n}{\lcm(j,c)}N_1^c =\sum_{d|v}u\phi(d)\sum_{\stackrel{\stackrel{\scriptstyle{c|n}}{\scriptstyle{d|c}}}
{\scriptstyle{j|\lcm(i,c)}}}\frac{n}{\lcm(j,c)}N_{1}^c,
\]
which is satisfied since
\[
\hskip 1.8in \sum_{d | i} \phi(d) = i = uv = \sum_{d | v} u \phi(d). \esinglebox
\]
\end{proof*}
\begin{definition*}
We define $L$ to be the matrix with rows indexed by pairs from $\{(i,j):j|n,\ i|j\}$ and columns by pairs from
$\{(d,c):c|n,\ d|c\}$, such that each entry $L_{(i,j)}^{(d,c)}$ is the coefficient of~$N_d^c$ in~$L_i^j$.
Similarly, we define $R$ to be the matrix whose entry $R_{(i,j)}^{(d,c)}$ is the coefficient of~$N_d^c$ in~$R_i^j$.
\end{definition*}
Proposition \ref{e1} gives us a subspace of the kernel of $L-R$ whose dimension is the number of divisors
$\tau(n)$ of $n$. To establish Theorem~\ref{MT} it will suffice to show that this is in fact the full kernel.

\begin{proposition}\label{Linvertible}
The matrix $L$ is invertible.
\end{proposition}
\begin{proof} We show that $L$ has non-zero determinant. Let $S$ be the group of permutations of the set
$\{(i,j):j|n,\ i|j\}$. Then
\begin{equation}\label{detL}
\det L=\sum_{\sigma\in S}\sgn(\sigma)\prod_{(i,j)}L_{(i,j)}^{(i,j)\sigma},
\end{equation}
where $\sgn(\sigma)$ is the sign of $\sigma$. Now let $\sigma$ be a particular permutation, and consider the cycle
\[
(i,j)=(i_0,j_0)\stackrel{\sigma}{\displaystyle{\longmapsto}}(i_1,j_1)\stackrel{\sigma}{\displaystyle{\longmapsto}}
\cdots\stackrel{\sigma}{\displaystyle{\longmapsto}}(i_t,j_t)=(i,j).
\]
Suppose that $\sigma$ contributes non-trivially to the sum (\ref{detL}). Then the product
\[
\prod_{k=0}^{t-1}L_{(i_k,j_k)}^{(i_{k+1},j_{k+1})}
\]
must be non-zero. We see from the definition (\ref{L}) of $L_i^j$ that $L_{(i,j)}^{(d,c)}=0$
unless $d|i$.
So $i_{k+1}|i_k$ for all $k$, and it clearly follows that $i_0=i_1=\cdots=i_t=i$. Furthermore,
we see that $L_{(i,j)}^{(d,c)}=0$ unless $d|c$ and $j|\lcm(i,c)$. Now since we have shown
that the values $i_k$ in our cycle are equal, it follows that we require
$i_k|j_{k+1}$ for all $k$, and hence that $\lcm(i_k,j_{k+1})=j_{k+1}$. We therefore see that
$j_k|j_{k+1}$ for all $k$, and so $j_0=j_1=\cdots=j_t=j$. Clearly this implies that $\sigma$ is the
identity permutation, and hence that
\[
\det L=\prod_{(i,j)}L_{(i,j)}^{(i,j)}.
\]
Now we see from (\ref{L}) that the coefficient of $N_i^j$ in $L_i^j$ is $\phi(i)n/j$, and
it follows that $\det L$ is non-zero.\end{proof}

The nullity of the matrix $L-R$ is equal to that of $I-RL^{-1}$, and hence to the multiplicity of~$1$ as an
eigenvalue of $RL^{-1}$. By establishing lower bounds for the dimensions of the eigenspaces of $RL^{-1}$ for its
other eigenvalues, we shall establish $\tau(n)$ as an upper bound for this dimension; this will suffice to prove
Theorem \ref{MT}. In fact we shall eventually establish the following result, which characterizes the matrix
$RL^{-1}$ completely.
\begin{lemma}\label{RLinv}
$RL^{-1}$ is diagonalizable, and its characteristic polynomial is
\[
\prod_{d|n}\left(x-\frac{\mu(d)}{d}\right)^{\tau(n/d)}\!\!,
\]
where $\mu$ is the M\"obius function.
\end{lemma}

\section{Reduction to the prime-power case}\label{reduction}
To prove Lemma \ref{RLinv}, it will first be necessary to reduce the problem to the case where $n$ is a power of
a prime; we do this by means of a tensor factorization. For a given integer $n$, let
\[
\mathcal{A}(n)=\{A_x^y(n) :y|n,\ x|y\}
\]
 be a set of variables. Then if $r$ and $s$ are coprime
integers, we may identify the set $\mathcal{A}(rs)$ with the tensor product $\mathcal{A}(r)\otimes
\mathcal{A}(s)$ by identifying $A_x^y(r)\otimes A_w^z(s)$ with $A_{xw}^{yz}(rs)$.

Now our notation $N_{i}^{j}$ involves an implicit argument $n$. We shall
consider the expression~$X_i^j(n)$ obtained from~$L_i^j$ (given explicitly at (\ref{L}) above)
by replacing each~$N_d^c$ with the variable~$A_d^c(n)$. Let $m$ and $M$ be
coprime integers such that $mM=n$; suppose $j|m$, $i|j$, $J|M$, $I|J$. Then we have
\begin{eqnarray*}
X_{iI}^{jJ}(mM) & = & \sum_{d|iI}\phi(d)\sum_{\stackrel{\stackrel{\scriptstyle{c|mM}}{\scriptstyle{d|c}}}
{\scriptstyle{jJ|\lcm(iI,c)}}}\frac{mM}{\lcm(jJ,c)}A_d^c(mM)  \\
& = & \sum_{d|i}\sum_{D|I}
\;\sum_{\stackrel{\stackrel{\scriptstyle{c|m}}{\scriptstyle{d|c}}}{\scriptstyle{j|\lcm(i,c)}}}
\;\sum_{\stackrel{\stackrel{\scriptstyle{C|M}}{\scriptstyle{D|C}}}{\scriptstyle{J|\lcm(I,C)}}}
\frac{m\phi(d)A_d^c(m)}{\lcm(j,c)}\otimes \frac{M\phi(D)A_D^C(M)}{\lcm(J,C)} \\
& = & X_i^j(m)\otimes X_I^J(M).
\end{eqnarray*}

Now let $Y_i^j(n)$ be the expression obtained from $R_i^j$ (given at (\ref{R}) above)
by replacing each~$N_d^c$ with~$A_d^c(n)$. Then
an argument similar to that above shows that
\[
Y_{iI}^{jJ}(mM)=Y_i^j(m)\otimes Y_I^{J}(M).
\]

This tensor factorization transfers easily to the matrices $L$ and $R$; for the rest of this section we shall use the more
explicit notation $L(n)$ and $R(n)$ for these matrices. Then it is easy to see from the equations above that
(provided that rows and columns are suitably ordered),
\[
L(n)=L(m)\otimes L(M),\;\; R(n)=R(m)\otimes R(M),
\]
and it follows that
\[
R(n)L(n)^{-1}=R(m)L(m)^{-1}\otimes R(M)L(M)^{-1}.
\]

The reduction of Lemma \ref{RLinv} to the case where $n$ is a prime power is now straightforward. For we notice that
$R(n)L(n)^{-1}$ is diagonalizable if its tensor factors are;
its eigenvalues are the products of the
eigenvalues of the tensor factors, with corresponding multiplicities.
It is not difficult to see that if Lemma~\ref{RLinv} is true for prime powers then
the multiplicativity of the arithmetic functions $\mu$ and $\tau$ will ensure
that it is true for all $n$.

\section{Proof in the prime-power case}{\label{primepower}}
We consider the matrix $RL^{-1}$ in the case when $n$ is a prime power $p^a$. According to Lemma~\ref{RLinv}
(which we have to verify), this matrix should have precisely three eigenspaces: a $1$-eigenspace of dimension
$a+1$, a $(-1/p)$-eigenspace of dimension $a$, and a kernel with dimension
\[
\sum_{b=2}^a\tau(p^{a-b})=\frac{a(a-1)}{2}.
\]

We have already, in Proposition \ref{e1}, established the existence of 
an $(a+1)$-dimensional space of
eigenvectors with eigenvalue $1$. We shall show next that the kernel of $RL^{-1}$ has dimension at
least $a(a-1)/2$. Lastly we shall exhibit a set of $a$ linearly independent row eigenvectors for~$RL^{-1}$ 
with eigenvalue $-1/p$, 
thus completing the proof of Lemma~\ref{RLinv}.

\begin{proposition}\label{e0}
\[
\dim \ker RL^{-1}\ge {a\choose 2}=\frac{a(a-1)}{2}.
\]
\end{proposition}
\begin{proof} It is clearly sufficient to show that the nullity of $R$ is at least ${a\choose 2}$.
Choose $r$ and $s$ such that $0\le r<s\le a$. There is a row of $R$ corresponding to the expression
$R_{p^r}^{p^s}$ defined in~(\ref{R}); 
this expression involves the quantities $u$ and $v$. Since $v$ is the
largest divisor of $p^r$ coprime with~$p^s/p^r$, it is clearly equal to $1$; 
it follows that $u=p^r$. We derive the simplified formula
\[
R_{p^r}^{p^s} = p^r\sum_{b=s}^{a}p^{a-b}N_{p^s}^{p^b}.
\]
But the summation here does not depend on $r$; in fact, it is clear that
$R_{p^r}^{p^s}=p^r\times R_{1}^{p^s}$. Thus each of the ${a \choose 2}$ rows
indexed by values $r,s$ such that $1 \le r < s \le a$ is a linear multiple of a row for which $r=0$.
This establishes the proposition.\end{proof}

We note that the eigenspace of $RL^{-1}$ for the eigenvalue $-1/p$ is equal to the kernel of $pR+L$.
\begin{proposition}\label{ep}
\[
\dim \ker (pR+L) \ge a.
\]
\end{proposition}
\begin{proof} Recall that both the rows and the columns of our matrix are indexed by pairs $(i,j)$ such that
$j|p^a$ and $i|j$. For $b\in\{0,\dots,a-1\}$, define $w^b$  to be the row vector whose
$(i,j)$th entry is given by
\[
\left\{\begin{array}{cl}
1 & \textrm{if $i=j=p^b$,}\\
p^2-1 & \textrm{if $i=p^b$ and $j=p^{b+1}$,}\\
-p & \textrm{if $i=j=p^{b+1}$,}\\
0 & \textrm{otherwise.}
\end{array}\right.
\]
It is clear that the vectors $w^b$ are linearly independent, since for each $b$ the $(p^b,p^{b+1})$th
coefficient is $p^2-1$ for $w^b$ and $0$ for $w^c$ whenever $c\neq b$.

To show that $w^b (pR+L) = 0$, we need to show that
\[
(pR_{p^b}^{p^b}+L_{p^b}^{p^b})+ (p^2-1)(pR_{p^b}^{p^{b+1}}+L_{p^b}^{p^{b+1}})-p(pR_{p^{b+1}}^{p^{b+1}}
+L_{p^{b+1}}^{p^{b+1}})=0.
\]
The expressions $R_{p^k}^{p^k}$ and $L_{p^k}^{p^k}$ are identical for all $k$, and thus the left-hand-side of
the equation above becomes
\begin{equation}\label{equ}
(p+1)L_{p^b}^{p^b}+ (p^2-1)(pR_{p^b}^{p^{b+1}}+L_{p^b}^{p^{b+1}})-p(p+1)L_{p^{b+1}}^{p^{b+1}}=0.
\end{equation}
From (\ref{L}) and (\ref{R}) we obtain the following equations:
\begin{align*}
L_{p^b}^{p^b} &= \sum_{k=0}^{b}\phi(p^k)\left(\sum_{l=k}^{b}p^{a-b}N_{p^k}^{p^l}+
\sum_{l=b+1}^{a}p^{a-l}N_{p^k}^{p^l}\right), \\
L_{p^b}^{p^{b+1}} &= \sum_{k=0}^{b}\phi(p^k)\sum_{l=b+1}^{a}p^{a-l}N_{p^k}^{p^l}, \\
R_{p^b}^{p^{b+1}} &= p^b\sum_{l=b+1}^{a}p^{a-l}N_{p^{b+1}}^{p^l}.
\end{align*}
We consider the coefficient of $N_{p^k}^{p^l}$ in (\ref{equ}) in various cases:
\smallskip
\begin{enumerate}
\item If $0\le k\le l\le b$ then $N_{p^k}^{p^l}$ occurs in $L_{p^b}^{p^b}$ and in
$L_{p^{b+1}}^{p^{b+1}}$, but nowhere else in (\ref{equ}). Its coefficient in (\ref{equ}) is
\[
(p+1)\phi(p^k)p^{a-b}-p(p+1)\phi(p^k)p^{a-b-1}=0.
\]
\item If $0\le k\le b< l$, then $N_{p^k}^{p^l}$ occurs in $L_{p^b}^{p^b}$,
$L_{p^{b+1}}^{p^{b+1}}$ and in $L_{p^{b}}^{p^{b+1}}$\!\!. Its coefficient in (\ref{equ}) is
\[
(p+1)\phi(p^k)p^{a-l}- p(p+1)\phi(p^k)p^{a-l}+ (p^2-1)\phi(p^k)p^{a-l}=0.
\]
\item If $b+1 = k \le l$ then $N_{p^k}^{p^l}$ occurs in $L_{p^{b+1}}^{p^{b+1}}$ and in
$R_{p^{b}}^{p^{b+1}}$\!\!. Its coefficient in (\ref{equ}) is
\[
-p(p+1)\phi(p^{b+1})p^{a-l}+(p^2-1)p^{b+1}p^{a-l}=0,
\]
since $\phi(p^{b+1})=(p-1)p^b$.
\end{enumerate}
\smallskip

\noindent Since the three cases considered here exhaust all of the variables $N_{p^k}^{p^l}$ which occur in
(\ref{equ}), this completes the proof of the proposition. \end{proof}

We have now completed the proof of Theorem~\ref{MT}. We remind the reader 
of the structure of the
proof: Propositions~\ref{e1},~\ref{e0} and~\ref{ep} together imply the correctness of Lemma~\ref{RLinv} in the
case where $|G/H|$ is a prime power. By the argument of Section~\ref{reduction}, this is sufficient to establish
it in the general case. It follows that Proposition~\ref{e1} describes the general solution to the system of
equations~$\Omega_i^j$ defined at~(\ref{Omega}), and this suffices to prove~Theorem \ref{MT}.

\section{Application: Counting the conjugacy classes of a subgroup}\label{application}

As an application of Theorem \ref{MT}, we find expressions for the number of conjugacy classes of the subgroup
$K_d$ under its own action, in terms of quantities relating only to the action of $G$. We first define these
quantities.
\begin{definition*} We define
\begin{enumerate} 
\item $T_d$ to be the number of $G$-conjugacy classes in the 
coset $\Gamma_d$;
\item $S_d$ to be the number of $G$-conjugacy classes in the subgroup $K_d$;
\item $S^*_d$ to be the number of conjugacy classes of $K_d$
under its own action.
\end{enumerate}
\end{definition*}

Now we have the following theorem.

\begin{theorem}\label{App}
\smallskip
\begin{enumerate}
\item $\displaystyle{S^*_d=n\sum_{c|n}\sum_{a|c}\mu(c/a)\frac{\hcf(a,d)}{\lcm(a,d)}T_c}$,\smallskip
\item $\displaystyle{S^*_d=n\sum_{c|n}\sum_{a|c}\sum_{b|c}\frac{\mu(c/a)\mu(c/b)}{\phi(c)}
\frac{\hcf(a,d)}{\lcm(a,d)}S_b}$.
\end{enumerate}
\end{theorem}
\begin{proof} Let $X$ be a $G$-conjugacy class inside $K_d$, and suppose that the centralizing subgroup of $X$ with
respect to~$H$ is~$K_c$.
Then the number of classes into which $X$ splits under conjugacy by $K_d$ is $n/\lcm(c,d)$. For each divisor
$b$ of $d$ there are $\phi(b)$ cosets of $H$ in $K_d$ whose order in the quotient group $K_d/H$ is $b$, and it follows from
these facts that
\[
S^*_d=\sum_{b|d}\phi(b)\sum_{\stackrel{\scriptstyle{c|n}}{\scriptstyle{b|c}}}\frac{n}{\lcm(c,d)}N_b^c=
\sum_{c|n}\sum_{b|\hcf(c,d)}\frac{\phi(b)n}{\lcm(c,d)}N_1^c,
\]
in which we have used Theorem \ref{MT} to replace $N_b^c$ with $N_1^c$. Now
\[
\sum_{b|\hcf(c,d)}\phi(b)=\hcf(c,d),
\]
and so
\[
S^*_d=\sum_{c|n}n\frac{\hcf(c,d)}{\lcm(c,d)}N_1^c.
\]

Now it is clear that
\[
T_b=\sum_{\stackrel{\scriptstyle{c|n}}{\scriptstyle{b|c}}}N_b^c=
\sum_{\stackrel{\scriptstyle{c|n}}{\scriptstyle{b|c}}}N_1^c,
\]
where here again we have invoked Theorem \ref{MT}. We look for quantities $\alpha_b$ such that
\mbox{$\sum_{b|n}\alpha_bT_b=S^*_d$.} 
We require that
\[
\sum_{c|n}\sum_{b|c}\alpha_bN_1^c = \sum_{c|n}n\frac{\hcf(c,d)}{\lcm(c,d)}N_1^c, \]
and clearly this implies that
\[
\sum_{b|c}\alpha_b=n\frac{\hcf(c,d)}{\lcm(c,d)}.
\]
By the M\"obius inversion formula (treating $d$ as constant) this condition is satisfied if and only if
\[
\alpha_b=n\sum_{a|b}\mu(b/a)\frac{\hcf(a,d)}{\lcm(a,d)},
\]
and the first part of the theorem now follows easily.

Finally, we observe that $S_d=\sum_{b|d}\phi(b)T_b$, and so by M\"obius inversion we have
\[
\phi(c)T_c=\sum_{b|c}\mu(c/b)S_b.
\]
By means of a simple substitution it is now easy to derive the second part of the theorem from the first.
\end{proof}

\affiliationone{
John R. Britnell, \\
School of Mathematics and Statistics, \\
Newcastle University, \\
Newcastle upon Tyne, \\
NE1 7RU, \\
UK
\email{j.r.britnell@ncl.ac.uk}
}
\affiliationtwo{
Mark Wildon, \\
Department of Mathematics, \\
Swansea University, \\
Swansea, \\
SA2 8PP, \\
UK
\email{m.j.wildon@swansea.ac.uk}
}

\end{document}